\documentclass[10pt,twoside, reqno]{amsart}
\usepackage[matrix,arrow]{xy}
\usepackage{amssymb}
\usepackage{amsmath}
\usepackage{amsfonts}
\usepackage{epsfig}
\usepackage{color}
\usepackage{epstopdf}
\usepackage{graphicx}
\usepackage{amsthm}
\usepackage{enumerate}
\usepackage[mathscr]{eucal}
\usepackage{verbatim}
\usepackage{bm}

\numberwithin{equation}{section}

\usepackage[bookmarks=true,hyperindex,pdftex,colorlinks,citecolor=blue,
linkcolor=blue,urlcolor=blue]{hyperref}
\usepackage{tikz}
\usetikzlibrary{matrix,arrows.meta}

\newcommand{\vertiii}[1]{{\left\vert\kern-0.25ex\left\vert\kern-0.25ex\left\vert #1 
		\right\vert\kern-0.25ex\right\vert\kern-0.25ex\right\vert}}

\theoremstyle{plain}
\newtheorem{theorem}{Theorem}[section]

\theoremstyle{definition}

\newtheorem{remark}[theorem]{Remark}

\begin{document}
	\title{Weak Minimizing Property and  the Compact Perturbation Property for the Minimum Modulus} 
	
\author[A. Raposo Jr.]{Anselmo Raposo Jr.}
\address[Anselmo Raposo Jr.]{Departamento de Matem\'{a}tica \newline\indent
	Universidade Federal do Maranh\~{a}o \newline\indent
	65085-580 - S\~{a}o Lu\'is, Brazil.}
\email{\href{mailto:anselmo.junior@ufma.br}{anselmo.junior@ufma.br}}

\author[G. Ribeiro]{Geivison Ribeiro}
\address[Geivison Ribeiro]{Departamento de Matem\'{a}tica \newline\indent
	Universidade Federal do Maranh\~{a}o \newline\indent
	65085-580 - S\~{a}o Lu\'is, Brazil.}
\email{\href{mailto:geivison.ribeiro@ufma.br}{geivison.ribeiro@ufma.br} \textrm{and} \href{mailto:geivison.ribeiro@academico.ufpb.br} {geivison.ribeiro@academico.ufpb.br}}

	
	\begin{abstract}
		For an operator $T:X\to Y$, denote $m(T)=\inf\{\|Tx\|:x\in S_X\}$. 
		A sequence $(x_n)$ in $S_X$ is said to be minimizing for $T$ if $\|Tx_n\|\to m(T)$. 
		The weak minimizing property (WmP), introduced by Chakraborty, requires that every operator admitting a non-weakly null minimizing sequence attains its minimum modulus.
		
		More recently, Han~\cite{Han2026} introduced the Compact Perturbation Property for the minimum modulus (CPPm), which requires that for every operator $T:X\to Y$ that does not attain its minimum modulus,
		\[
		\sup_{K\in\mathcal{K}(X,Y)} m(T+K)=m(T).
		\]
		
		In~\cite{Han2026}, it is shown that $(\ell_1,\ell_1)$ fails both properties, while $(c_0,c_0)$ fails the WmP. However, whether $(c_0,c_0)$ has the CPPm was left open (Problem~3.6).
		
		In this paper, we give a negative answer to this question by proving that $(c_0,c_0)$ does not have the CPPm. The proof is constructive, exhibiting a non-min-attaining operator whose minimum modulus is strictly increased by a rank-one compact perturbation.
		
		Moreover, we show that this phenomenon is not specific to $c_0$: if $X=\mathbb{K}\oplus_\infty Y$ with $Y$ non-reflexive, then the pair $(X,X)$ fails the CPPm. 
	\end{abstract}
	
	\maketitle
	
\section{Introduction}

For Banach spaces $X$ and $Y$, denote by $L(X,Y)$ the space of all bounded linear operators $T:X\to Y$. For each $T\in L(X,Y)$, consider the function
\[
\phi_T(x)=\|Tx\|,\qquad x\in S_X,
\]
defined on the unit sphere $S_X$ of $X$. While $\sup_{x\in S_X}\phi_T(x)=\|T\|$, the corresponding infimum
\begin{equation}\label{m}
	m(T)=\inf\{\|Tx\|:x\in S_X\}
\end{equation}
is known as the \emph{minimum modulus} of $T$.

The study of extremal properties of $\phi_T$ has a long tradition in functional analysis. Norm-attaining operators have been extensively investigated since the classical Bishop--Phelps theorem \cite{bis-phe}, and their density and structural properties have been studied by many authors, see for instance \cite{Ac2006,Diestel}. In contrast, the behavior of operators attaining their minimum modulus is less understood, partly because every non-injective operator trivially attains its minimum.

To overcome this difficulty, Chakraborty \cite{Cha2020} introduced the \emph{weak minimizing property} (WmP), which ensures the existence of minimizers under a weak compactness-type condition on minimizing sequences. This notion is a counterpart of the weak maximizing property (WMP) studied in \cite{AGPT}. A systematic study of WmP was carried out in \cite{Cha2020,Han2026, Ribeiro_Kadets}, where pairs of Banach spaces were analyzed.

In particular, it follows from general reflexivity considerations (see, e.g., \cite{AGPT} and related results) that $(c_0,c_0)$ fails the WmP. This raises the question of whether weaker stability properties for the minimum modulus may still hold in such spaces.

Motivated by this problem, Han \cite{Han2026} introduced the \emph{Compact Perturbation Property for the minimum modulus} (CPPm). A pair $(X,Y)$ is said to have the CPPm if for every operator $T\in L(X,Y)$ that does not attain its minimum modulus, one has
\begin{equation}\label{cppm}
	\sup_{K\in \mathcal{K}(X,Y)} m(T+K)=m(T),
\end{equation}
where $\mathcal{K}(X,Y)$ denotes the space of compact operators.

The CPPm can be viewed as a perturbative stability property for the minimum modulus. In \cite{Han2026}, it is shown that $(\ell_1,\ell_1)$ fails both the WmP and the CPPm, while the case of $(c_0,c_0)$ remains unresolved. More precisely, Han posed the following question (see \cite[Question~3.6]{Han2026}):

\medskip

\begin{quote}
	\textbf{Question.} Does the pair $(c_0,c_0)$ have the CPPm?
\end{quote}

\medskip

The main purpose of this paper is to answer this question in the negative.

\medskip

\begin{theorem}\label{main}
	The pair $(c_0,c_0)$ does not have the CPPm.
\end{theorem}

\medskip

Our proof is constructive: we exhibit an operator $T:c_0\to c_0$ that does not attain its minimum modulus and a rank-one compact operator $K$ such that $m(T+K)>m(T)$.

\medskip

Moreover, we show that this phenomenon is not specific to $c_0$, but reflects a general structural obstruction. Namely, if $X=\mathbb{K}\oplus_\infty Y$ with $Y$ non-reflexive, then $(X,X)$ fails the CPPm.

\medskip

\section{The results}\label{sec:mainproof}

In this section we prove the main result of the paper and its structural extension.

\medskip

\begin{proof}[Proof of Theorem~\ref{main}]
	Let $f\in (c_0)^*=\ell_1$ be a functional such that $\|f\|=1$, $f(e_1)=0$, and whose support is infinite. For instance, one may take
	\[
	f(x)=\sum_{j=2}^\infty 2^{1-j}x_j, \qquad x\in c_0.
	\]
	Define $T:c_0\to c_0$ by
	\begin{equation}\label{defT}
		Tx=x-f(x)e_1.
	\end{equation}
	
	Let $x\in S_{c_0}$ and set $s=\sup_{j\ge 2}|x_j|$. Since $f$ depends only on the coordinates $j\ge 2$, we have $|f(x)|\le s$, and therefore
	\[
	\|Tx\|=\max\{|x_1-f(x)|,s\}.
	\]
	If $s\ge \frac12$, then $\|Tx\|\ge s\ge \frac12$. If $s<\frac12$, then $|x_1|=1$, and hence
	\[
	|x_1-f(x)|\ge |x_1|-|f(x)|\ge 1-s>\frac12.
	\]
	Thus $\|Tx\|\ge \frac12$ for every $x\in S_{c_0}$, so $m(T)\ge \frac12$.
	
	For the reverse inequality, consider
	\[
	x^{(N)}=e_1+\frac12\sum_{j=2}^N e_j.
	\]
	Then $\|x^{(N)}\|=1$ and
	\[
	f(x^{(N)})=\frac12\sum_{j=2}^N2^{1-j}\longrightarrow \frac12.
	\]
	Consequently,
	\[
	\|Tx^{(N)}\|
	=
	\max\left\{|1-f(x^{(N)})|,\frac12\right\}
	\longrightarrow \frac12.
	\]
	Hence
	\[
	m(T)=\frac12.
	\]
	
	We claim that $T$ does not attain its minimum modulus. Suppose, towards a contradiction, that there exists $x\in S_{c_0}$ such that $\|Tx\|=\frac12$. Then $s\le \frac12$ and, since $\|x\|=1$, we have $|x_1|=1$. Multiplying $x$ by a scalar of modulus one, if necessary, we may assume that $x_1=1$. Since $T$ is linear, this does not change the value of $\|Tx\|$.
	
	Thus
	\[
	|1-f(x)|\le \frac12,
	\]
	and therefore
	\[
	|f(x)|\ge 1-|1-f(x)|\ge \frac12.
	\]
	On the other hand,
	\[
	|f(x)|
	\le \sum_{j=2}^{\infty}2^{1-j}|x_j|
	\le \frac12\sum_{j=2}^{\infty}2^{1-j}
	=
	\frac12.
	\]
	Hence $|f(x)|=\frac12$. Equality in the above estimate forces
	\[
	|x_j|=\frac12
	\qquad (j\ge 2).
	\]
	This contradicts $x\in c_0$. Therefore, $T$ does not attain its minimum modulus.
	
	Finally, define $K:c_0\to c_0$ by
	\[
	Kx=f(x)e_1.
	\]
	Then $K$ is compact (has rank one), and
	\[
	(T+K)x=x \qquad (x\in c_0).
	\]
	Thus $T+K=I_{c_0}$, and consequently
	\[
	m(T+K)=1>\frac12=m(T).
	\]
	This shows that the pair$(c_0,c_0)$ does not have the CPPm.
\end{proof}

\medskip

We now show that the above construction is not specific to $c_0$.

\medskip

\begin{theorem}\label{general}
	Let $X=\mathbb{K}\oplus_\infty Y$, where $Y$ is a non-reflexive Banach space. Then the pair $(X,X)$ does not have the CPPm.
\end{theorem}

\begin{proof}
	Since $Y$ is non-reflexive, James' theorem gives a functional $f\in S_{Y^*}$ that does not attain its norm. Define $T:X\to X$ by
	\[
	T(a,y)=(a-f(y),y).
	\]
	
	Let $(a,y)\in S_X$ and put $s=\|y\|$. Then
	\[
	\|T(a,y)\|=\max\{|a-f(y)|,\|y\|\}.
	\]
	If $s\ge \frac12$, then $\|T(a,y)\|\ge \frac12$. If $s<\frac12$, then $|a|=1$, and therefore
	\[
	|a-f(y)|\ge |a|-|f(y)|\ge 1-\|y\|>\frac12.
	\]
	Thus $m(T)\ge \frac12$.
	
	Since $\|f\|=1$, there exists $(y_n)\subset S_Y$ such that $|f(y_n)|\to 1$. Multiplying by suitable scalars of modulus one, we may assume that $f(y_n)\to 1$. For
	\[
	x_n=\left(1,\frac12 y_n\right)
	\]
	we have $\|x_n\|=1$ and
	\[
	\|Tx_n\|
	=
	\max\left\{\left|1-\frac12 f(y_n)\right|,\frac12\right\}
	\longrightarrow \frac12.
	\]
	Hence $m(T)=\frac12$.
	
	If $T$ attained its minimum modulus at some $(a,y)\in S_X$, then
	\[
	\max\{|a-f(y)|,\|y\|\}=\frac12.
	\]
	Thus $\|y\|\le \frac12$ and $|a|=1$. Also,
	\[
	|f(y)|\ge |a|-|a-f(y)|\ge \frac12.
	\]
Furthermore
	\[
	|f(y)|\le \|y\|\le \frac12.
	\]
	Hence $|f(y)|=\|y\|=1/2$, and therefore
	\[
	\left|f\left(\frac{y}{\|y\|}\right)\right|=1,
	\]
	contradicting the fact that $f$ does not attain its norm. Thus $T$ does not attain its minimum modulus.
	
	Finally, define $K:X\to X$ by
	\[
	K(a,y)=(f(y),0).
	\]
	Then $K$ has rank one and
	\[
	(T+K)(a,y)=(a,y).
	\]
	Hence $T+K=I_X$, so
	\[
	m(T+K)=1>\frac12=m(T).
	\]
	Therefore the pair $(X,X)$ does not have the CPPm.
\end{proof}

\begin{remark}
	The case $X=c_0$ also follows from Theorem~\ref{general}, since
	$c_0$ is isometric to $\mathbb K\oplus_\infty c_0$.
\end{remark}

\section*{Acknowledgments}

The authors thank V.~Kadets for suggesting the idea leading to the generalization presented in this paper. The author A.~Raposo. was partially supported by the projects CNPq 406457/2023-9 and 302341/2025-0.

\end{document}